\documentclass[preprint,pdftex]{amsart}
\usepackage{amsfonts,amssymb,amsmath,latexsym,enumerate,array,mathrsfs,color,enumerate}
\usepackage{amscd}     
\usepackage{graphicx}
\newtheorem{thm}{Theorem}
\newtheorem{lem}[thm]{Lemma}
 
\newtheorem{cor}[thm]{Corollary} \newtheorem{prop}[thm]{Proposition}

\newtheorem{nota}{Notation}
\newtheorem{conj}{Conjecture}

\newtheorem{ex}[thm]{Example} \newtheorem{defi}[thm]{Definition}
\newtheorem{rem}{Remark}
\newtheorem{rems*}{Remarks}
\newtheorem{problem*}{Problem}
\newtheorem{conj*}{Conjecture}
\newcommand{\sO}{\mathscr{O}}

\newcommand{\cI}{\mathcal{I}}

 \DeclareMathOperator{\hh}{h}

\newcommand{\Z}{\mathbb Z} \newcommand{\C}{\mathbb K}

 \newcommand{\p}{\mathbb P}

\numberwithin{equation}{section}

\begin{document}
\title{Laplace equations, Lefschetz properties and line arrangements}
\author[rdg]{Roberta Di Gennaro
}
\author[gi]{Giovanna Ilardi}
\date{}
\begin{abstract}
In this note we extend the main result in \cite{DIV} on artinian ideals failing Lefschetz properties, varieties satisfying Laplace equations and existence of suitable singular hypersurfaces. Moreover we characterize the minimal generation of ideals generated by powers of linear forms by the configuration of their dual points in the projective plane and we use this result to improve some propositions on line arrangements and Strong Lefschetz Property at range $2$ in \cite{DIV}. The starting point was an example in \cite{CHMN16}. Finally we show the equivalence among failing SLP, Laplace equations and some unexpected curves introduced in \cite{CHMN16}.
\end{abstract}\maketitle
\section{Introduction}
Recently there has been an increasing interest both on ideals failing Lefschetz Pro\-perties and on arrangements of hyperplanes. In \cite{DIV} the authors link these two topics. In fact, given a suitable line arrangement in the projective plane, i.e. a finite collection of lines, they relate the unstability of its derivation bundle (see Section \ref{arrang_sec}) to the failure of the Strong Lefschetz Property at range 2  of the ideal generated by suitable powers of the linear forms defining the lines (see Section \ref{lef_sec}). This linkage is deepened in \cite{CHMN16} where the authors add the equivalent condition of existence of a suitable unexpected curve.
\\ Here we recover these arguments in order to improve and generalize \cite{DIV}. Precisely, in Section \ref{lef_sec} we generalize the main theorem in \cite{DIV} to the case when there exist syzygies of suitable degree (Theorem \ref{thgen}), and in Section \ref{arrang_sec} we reformulate and improve some results in \cite{DIV}. More precisely, given an artinian ideal $I$ ge\-nerated by powers of linear forms, we characterize geometrically the existence of syzygies of degree $0$. Let $Z$ be
the set of points which are dual to these linear forms, then the existence of $0$-syzygies in $I$ is related to the number of aligned points in $Z$ (Theorem \ref{th_aligned}). In this way we restate some results (Proposition \ref{prop_bundle} and Conjecture \ref{conj_terao}) in a more geometric form. The note ends with the equivalence among the existence of unexpected curves, the failure of SLP at range 2 and the existence of Laplace equations in suitable cases (Corollary \ref{cor_unexp}).

\section{Lefschetz properties}\label{lef_sec}
Let $\mathbb K$ be an algebraically closed field of characteristic $0$, $R=\mathbb K[x_0,\ldots, x_n]=\bigoplus R_t$ be the graded polynomial ring in $n+1$ variables over $\mathbb K$ and $r_t= \dim R_t=\dim H^0(\mathscr O_{\p^n}(t))=\left( \begin{array}{c}t+n\\n\end{array}\right)$, when $n$ is fixed.
\\Let
$$A=R/I= \bigoplus_{i=0}^{m}A_i$$ be a graded artinian algebra, defined by a homogeneous ideal $I$. Note that $A$ is finite dimensional over $\mathbb K$. We denote by $H_A=H_{R/I}$ the Hilbert function of $R/I$.
\\First of all, we recall the definition of algebra (or ideal) failing the Lefschetz Properties.
\begin{defi}
 The artinian algebra $A$ (or the artinian ideal $I$) fails the Weak Lefschetz Property (from now on WLP) if  for any linear form $L$   there exists $i$  such that the multiplication map by $L$,
$ \times L : A_i \rightarrow A_{i+1},$
 does not have maximal rank (i.e. is neither injective nor surjective). More precisely, $A$ (or $I$) fails WLP \em by $\delta$\em\ if the multiplication map has rank $\min\{H_{R/I}(i),H_{R/I}(i+1)\}-\delta$(see \cite{CN13}).\end{defi}
Similarly, we define the failure of the Strong Lefschetz Property by $\delta$.
\begin{defi} The artinian algebra $A$ (or the artinian ideal $I$)  fails the Strong
Lefschetz Property (from now on SLP) at range $k$ and degree $i$ by $\delta$ (with $\delta \geq 1$) if  for any linear form $L$  the multiplication map $ \times L^k : A_i \rightarrow A_{i+k},$ has rank $\min\{H_{R/I}(i),H_{R/I}(i+k)\}-\delta$.
\end{defi}
One of the main examples comes from a classical result of Togliatti: the ideal $I=(x^3,y^3,z^3,xyz)$ fails the WLP in degree $2$ by $1$. In \cite[Example 3.1]{BK} this ideal is studied, but it appeared in \cite{T} in terms of projection center of the Veronese surface (see also \cite{DI} for a modern approach and below for some details).
\\
In \cite{DIV} the authors of this note and J. Vall\`{e}s characterize artinian ideals failing SLP at range $k$ by $\delta$ in terms of suitable projections of Veronese varieties satisfying Laplace equations and of the existence of suitable singular hypersurfaces. In order to state our main result we recall some definitions and results from \cite{DIV}.
\begin{defi}
Let $I=(F_1, \ldots, F_r) \subset R$ be an artinian ideal generated by
forms of degree $d$. The syzygy bundle $K$ is defined by the exact sequence
$$  \begin{CD}
 0@>>> K @>>>  \mathscr O _{\p^{n}}^{r} @>\Phi_{I}>>
 \mathscr O _{\p^{n}}(d) @>>> 0,
\end{CD}$$
where $\Phi_{I}(a_1,\ldots,a_r)=a_1F_1+\ldots+a_rF_r.$
\end{defi}
\begin{thm}\cite[Theorem 4.1]{DIV}
\label{p1}
 Let $I=(F_1, \ldots, F_r) \subset R$ be an artinian ideal generated by
forms of degree $d$ and  $K$ the syzygy bundle.
Let $i$ be a non-negative integer such that $ \mathrm{h}^0( K(i))=0$ and $k$ be an integer such that $k\ge 1$.
Then $I$ fails the SLP at the range  $k$ in degree $d+ i-k$  if and only if the induced homomorphism on global sections
(denoted by $\mathrm{H}^0(\Phi_{I,L^k})$)
$$  \begin{CD}
  \mathrm{H}^0(  \mathscr O _{L^k}(i))^r @>\mathrm{H}^0(\Phi_{I,L^k})>>
 \mathrm{H}^0(  \mathscr O _{L^k}(i+d))
\end{CD}$$
does not have maximal rank for a general linear form $L$.
\end{thm}

\begin{defi}
 Let $X\subset \mathbb P^N$ be a projective $n$-dimensional complex variety. For $m\geq 1$, the projective $m$-th osculating space to $X$ at a general point $P$, $T_P^{m}(X)$, is the subspace of $ \mathbb P^N$ spanned by $P$ and by all the derivative points of order less  than or equal to $m$ of a local parametrization of $X$, evaluated at $P$. Of course, for $m=1$ we get the tangent space $T_P(X)$.
\end{defi}
\noindent We remark that the expected dimension of the $m$-th osculating space is $$ \mathrm{expdim} T_P^m(X)=\mathrm{inf}\left(\binom{n+m}{n}-1,N\right).$$
\begin{defi}A $n$-dimensional variety $X\subset \mathbb P^N$  satisfies  $ \delta$ independent Laplace equations of order $m$ if the $m$-th osculating space at a general point has $$\dim T_P^{m}(X)=
\mathrm{expdim} T_P^m(X)-\delta.$$\\
If $N < \binom{n+m}{n}-1$, then there are always $\binom{n+m}{n}-1-N$ relations between the partial derivatives. We call these relations \lq\lq trivial\rq\rq\ Laplace equations of order $m$. \end{defi}
\par\vspace{3mm} Among the varieties satisfying non trivial Laplace equations, we are inte\-rested in rational varieties which are suitable projections of a Veronese variety.
\\ For any vector space $V$, $V^*=\mathrm{Hom}_{\C}(V,\C)$ will be the dual space.\\
 Let $v_{t} : [L] \in \mathbb P (R_1^*)\hookrightarrow [L^{t}]\in \mathbb P (R_{t}^*)$
be the $t$-uple Veronese embedding whose image  $v_t(\mathbb P^n)$ is the Veronese $n$-fold of order $t$.
\\Let $I=(F_1, \ldots, F_r)$  be an ideal generated by $r$ forms of degree $d$ and $I_h$ be the homogeneous component of degree $h$ of $I$ for any $h$.
\begin{defi}The apolar space of $I$ in degree $d+i$ with $i\geq 0$ is
$$  I_{d+i}^{\perp}=\{\Delta \in R_{d+i}^*\,|\,\Delta(F)=0, \,\, \forall F\in I_{d+i}\},$$ where the canonical basis of $R_{d+i}^*$ is given by the $r_{d+i}=\binom{d+i+n}n$ derivations $ \frac{\partial^{d+i}}{\partial x_0^{i_{0}}\ldots \partial x_n^{i_n}}$ with $ i_0+\ldots +i_n=d+i.$ \end{defi}
There is the exact sequence of vector spaces
$$  \begin{CD}
0 @>>> I_{d+i}^{\perp} @>>>R_{d+i}^* @>>> I_{d+i}^* @>>> 0
\end{CD}$$
and, by dualizing it, one can identify $R_{d+i}/I_{d+i}\simeq (I_{d+i}^{\perp})^*$ and write the decomposition
$R_{d+i}=I_{d+i} \oplus (I_{d+i}^{\perp})^*.$
\\ By denoting the corresponding  projection map
$$\begin{CD}
\pi_{I_{d+i}}: \mathbb P(R^*_{d+i})\setminus \mathbb P(I^*_{d+i}) \rightarrow \mathbb P(I^\perp _{d+i})
  \end{CD},
$$ we consider the variety $X:=\pi_{I_{d+i}}(v_{d+i}(\mathbb P^n))$.
\begin{rem}
The toric case is the easiest one: when $I_d$ is generated by $r$ monomials of degree $d$,   $(I_{d}^{\perp})^*$ is  generated by
the  remaining $r_d-r$ monomials.
\\ The case of powers of linear forms is very interesting. In \cite{EI} it is proved that the apolar of an ideal generated by powers of linear forms is related to the $0$-dimensional scheme of the dual points of the linear forms, as in the following Theorem.
\begin{thm}\cite{EI}\label{EI}
Let $l_1, \ldots, l_r$ be linear forms in $\mathbb P^n$, $d_1,\ldots,d_r >0$ be integers and $I=(l_1^{d_1},\ldots,l_r^{d_r})$. If $P_i=l_i^\vee$ denotes the dual point of $l_i$, then for any $j\geq \max{d_i}$,
$$\dim_{\mathbb K}\left(\frac R I\right)_j = \dim_{\mathbb K} \bigcap_{i=0,\ldots,r}I_{P_i}^{j-d_i+1}.$$
\end{thm}
\end{rem}
\begin{nota} For any $i\ge 0$, $k\ge 1$, we denote $N(r,i,k,d):=r(r_i-r_{i-k})- (r_{d+i}-r_{d+i-k})$,
$N^{+}=\mathrm{sup}(0,N(r,i,k,d))$ and $N^{-}=\mathrm{sup}(0,-N(r,i,k,d)).$
\end{nota}
\begin{rem}\label{remark}
Note that when $\mathrm h^0 (K(i))=0$ then $N(r,i,k,d)=H_{R/I}(d+i-k)-H_{R/I}(d+i)=\dim \ker (\times L^k) - \dim \mathrm{coker} (\times L^k)$.
\\Moreover if $N < \binom{n+d+i-k}{n}-1$ then the number of trivial equations is exactly
$N^+ + \dim I_{d+i-k}$.
\end{rem}
The main result in \cite{DIV} (that generalizes to SLP the result given in \cite{MMO} on WLP) is the following.
\begin{thm}\cite[Theorem 5.1]{DIV}
\label{th1bis}
Let $I=(F_1, \ldots, F_r) \subset R$ be an  artinian ideal generated by $r$ homogeneous polynomials  of degree $d$.
Let $i,k,\delta$ be integers such that $
i\ge 0$, $k\ge 1$.
Assume that there is no syzygy  of degree $i$  among the $F_j$'s.
The following conditions are equivalent:
\begin{enumerate}
\item The ideal $I$ fails the SLP at the range  $k$ in degree $d+ i-k$.
\item For a general linear form $L$ of $\p^n$, there exist  $N^{+}+\delta$, with $\delta \ge 1$,  independent vectors  $(G_{1j},\ldots, G_{rj})_{j=1, \ldots,N^{+}+\delta }\in R_i^{\oplus r}$ and
$N^{+}+\delta $ forms $G_j\in R_{d+ i-k}$ such that
$G_{1j}F_1+ \ldots + G_{rj}F_r=L^kG_j$.
\item
 The $n$-dimensional variety $\pi_{I_{d+i}}(v_{d+i}(\p^n))$ satisfies
$\delta \ge 1$ Laplace equations of order $d+i-k$.
\item \label{item_iv_thm} For any  $L\in R_1$, $\mathrm{dim}_{\C}((I_{d+i}^{\perp})^*\cap \mathrm{H}^0(\cI_{L^{\vee}}^{d+i-k+1}(d+i))\ge N^{-}+\delta$, with $\delta \ge  1$.
\end{enumerate}
\end{thm}
 As noted in \cite{DIV} the hypothesis on the global syzygy in Theorem \ref{th1bis} is not restrictive in case of syzygies of degree $i\geq 1$.
%
\begin{lem}\cite[Lemma 5.2]{DIV}
 \label{lem-syz}
Let $I$ be the ideal $(L_1^{d},\ldots,L_r^{d})$  where the $L_j$ are general linear forms and $r<r_d$. Let $K$ be its syzygy bundle and $i\geq 1$. Then
$$ \hh^0(K(i))=0 \Leftrightarrow rr_i\le r_{d+i}.$$
\end{lem}
 Here we generalize these results in the case in which $H^0(K(i))\neq 0$ and $H^0(K(i-k))= 0$. This generalization is interesting as if there are $0$-syzygies, i.e. $i=0$ and $H^0(K)\neq 0$ then the condition $H^0(K(-k))= 0$ is automatically verified.
\begin{thm}
\label{p1bis}
Let $i,k$ be non-negative integers such that $\mathrm{h}^0( K(i))=s$ and $H^0(K(i-k))= 0$.
Denoted by $\mathrm{H}^0(\Phi_{I,L^k})$ the induced homomorphism on global sections
$$  \begin{CD}
  \mathrm{H}^0(  \mathscr O _{L^k}(i))^r @>\mathrm{H}^0(\Phi_{I,L^k})>>
 \mathrm{H}^0(  \mathscr O _{L^k}(i+d))
\end{CD},$$
then the homomorphisms $\mathrm{H}^0(\Phi_{I,L^k})$ and  $\times L^k$ have the same cokernel, while $\ker (\times L^k)\cong \dfrac{\ker \mathrm{H}^0(\Phi_{I,L^k})}{\mathrm{H}^0( K(i))}$.
\end{thm}
\begin{proof} The proof exploits the same ideas as in \cite[Theorem 4.1]{DIV}.
Let us consider the canonical exact sequence
$$  \begin{CD}
 0@>>> K(i-k) @> \times L^k>> K(i) @>>>
 K\otimes \sO_{L^k}(i) @>>> 0.
\end{CD}$$
As $H^0(K(i-k))= 0$ and $A_{d+i}=\mathrm{H}^1(K(i))$ for any $i\in \Z$ (\cite[Proposition 2.1]{BK}), we obtain the long exact sequence of cohomology
$$
0 \rightarrow \mathrm{H}^0( K(i))\rightarrow
  \mathrm{H}^0(K\otimes \mathscr O _{L^k}(i)) \rightarrow A_{d+ i-k} \stackrel{\times L^k}\longrightarrow  A_{d+i} \rightarrow$$$$\rightarrow
\mathrm{H}^1(K\otimes \mathscr O _{L^k}(i))\rightarrow \mathrm{H}^2(K(i-k)) \rightarrow 0.
$$
Moreover, by tensoring the exact sequence defining the bundle $K$  by $\mathscr O _{L^k}(i)$
$$  \begin{CD}
 0@>>> K @>>>  \mathscr O _{\p^{n}}^{r} @>\Phi_{I}>>
 \mathscr O _{\p^{n}}(d) @>>> 0,
\end{CD}$$
we get in cohomology
$$ \begin{array}{ll} 0\longrightarrow &
  \mathrm{H}^0(K\otimes   \mathscr O _{L^k}(i)) \longrightarrow \mathrm{H}^0(  \mathscr O _{L^k}(i))^r \stackrel{\mathrm{H}^0(\Phi_{I,L^k})}\longrightarrow
 \mathrm{H}^0(  \mathscr O _{L^k}(i+d))\longrightarrow  \\ & \mathrm{H}^1(K\otimes   \mathscr O _{L^k}(i))\longrightarrow \mathrm{H}^1(  \mathscr O _{L^k}(i))^r \longrightarrow
 \mathrm{H}^1(  \mathscr O _{L^k}(i+d))
\longrightarrow 0.\end{array}$$
\\ Moreover $\mathrm{H}^2(K(i-k))=0=\mathrm{H}^1(\mathscr O _{L^k}(i))=0$ when $n>2$; while when $n=2$ we have $\mathrm{H}^2(K(i-k))=\mathbb K^t$ with $t= rr_{k-i-3} - r_{k-i-d-3}$ and
 $\mathrm{h}^1(\mathscr O _{L^k}(i)) =\mathrm{h}^2(\mathscr O _{\p^2}(i-k))=r_{k-i-3}$,  $\mathrm{h}^1(\mathscr O _{L^k}(i+d)) =\mathrm{h}^2(\mathscr O _{\p^2}(i+d-k))=r_{k-i-d-3}$;
 so the cokernel of both maps $\mathrm{H}^0(\Phi_{I,L^k}) $ and $\times L^k$ are the same and between the kernels there is the sought isomorphism.
\end{proof}
Following Remark \ref{remark}, it is natural to generalize the integers $N(r,i,k,d), N^+,N^-$ in case of syzygies.
\begin{defi}\label{def_Ns}
Let $h^0(K(i))=s$, we define:\begin{itemize} \item $N_s=N(r,i,k,d,s):=r(r_i-r_{i-k})- (r_{d+i}-r_{d+i-k})-s$; \item $N^{+}_s:=\mathrm{max}(0,N_s)$; \item $N^{-}_s:=\mathrm{max}(0,-N_s)).$\end{itemize}
\end{defi}
\begin{rem} Note that when $h^0(K(i-k))=0$ we have $N_s=H_{R/I}(d+i-k)-H_{R/I}(d+i)$.
\end{rem}
Now we give the main theorem of this note, that generalizes \cite[Theorem 5.1]{DIV}.
\begin{thm}
\label{thgen}
Let $I=(F_1, \ldots, F_r) \subset R$ be an  artinian ideal generated by $r$ homogeneous polynomials  of degree $d$.
Let $i,k,\delta$ be non negative integers such that there is no syzygy  of degree $i-k$  among the $F_j$'s.
The following conditions are equivalent:
\begin{enumerate}
\item \label{item_SLP}The ideal $I$ fails the SLP at the range  $k$ in degree $d+ i-k$ by $\delta$;
\item\label{item_ker} $\dim \ker (\times L^k)=N^+_s+\delta>N^+_s$;
\item\label{item_coker} $\dim \mathrm{coker} (\times L^k)=N^-_s+\delta>N^-_s$;
\item \label{item_equation}The $n$-dimensional variety $\pi_{I_{d+i}}(v_{d+i}(\p^n))$ satisfies
$\delta \ge 1$ non trivial Laplace equations of order $d+i-k$ and no Laplace equation of smaller order;
\item \label{item_hypersurf} For any  $L\in R_1$, $\mathrm{dim}_{\C}((I_{d+i}^{\perp})^*\cap \mathrm{H}^0(\cI_{L^{\vee}}^{d+i-k+1}(d+i))= N^{-}_s+\delta$, with $\delta \ge  1$.
\end{enumerate}
\end{thm}
\begin{proof}
By Theorem \ref{p1}, the equivalence between \eqref{item_ker} and \eqref{item_coker} is an easy calculation.
Moreover, it is obvious that \eqref{item_ker} and \eqref{item_coker} imply \eqref{item_SLP}.
\\ Let us assume that $I$ fails the SLP at the range  $k$ in degree $d+ i-k$. Let $D=\dim \ker (\times L^k)$ and suppose that $0<D\leq N^+_s$. Denoting $h^0=\dim \ker \mathrm{H}^0(\Phi_{I,L^k})$, by Theorem \ref{p1}, $D=h^0-s$ and $\dim \mathrm{coker} (\times L^k)=D-N_s$. Now, if $N_s\geq 0$, then $N^+_s=N_s=N-s$ and $0<D=h^0-s\leq N-s$ so $h^0\leq N$. On the other hand, $0<\dim \mathrm{coker} \times L^k=D-N_s=h^0-s-N+s=h^0-N$ so $h^0>N$ that is a contradiction.
If $N_s<0$, then $N^+_s=0$ and $0<D\leq 0$ is a contradiction. So \eqref{item_SLP} $\Leftrightarrow$ \eqref{item_ker} $\Leftrightarrow$ \eqref{item_coker}.
\\\eqref{item_ker} $\Leftrightarrow$ \eqref{item_equation}.
 The dimension of the kernel of the map $\times L^k$ i.e. the dimension of $\mathrm{H}^0(K\otimes \mathscr O _{L^k}(i))-s$, written in a geometric way, is
$$N^+_s+\delta=\mathrm{dim}[\p(I_{d+i}^*) \cap T_{[L^{d+i}]}^{d+i-k}v_{d+i}(\p^n)]+1,\,\,\,\,\, \delta \ge 0$$
where the projective dimension is $-1$ if the intersection is empty.  The number $\delta$ is the number of (non trivial) Laplace equations. Indeed,
 the dimension of the $(d+i-k)$-th osculating space to $\pi_{I_{d+i}}(v_{d+i}(\p^n))$ is $r_{d+i-k} -N^{+}_s -\delta$, since
 the $(d+i-k)$-th osculating space to $v_{d+i}(\p^n)$  meets the center of projection along a $\p^{N^{+}_s+\delta-1}$.  In other words, the $n$-dimensional variety  $\pi_{I_{d+i}}(v_{d+i}(\p^n))$
 satisfies $\delta$ Laplace equations.
\\ \eqref{item_equation}$\Leftrightarrow$ \eqref{item_hypersurf}.
The image  by $\pi_{I_{d+i}}$ of the  $(d+i-k)$-th osculating space to the Veronese
$v_{d+i}(\p^n)$ at a general point  has codimension $\mathrm{h}^0(K\otimes \mathscr O _{L^k}(i))-N^+_s$ in
$\p(I_{d+i}^{\perp})$. The codimension corresponds to the number of hyperplanes in $\p(I_{d+i}^{\perp})$ containing the osculating space to
$\pi_{I_{d+i}}(v_{d+i}(\p^n))$. These hyperplanes are images by $\pi_{I_{d+i}}$ of hyperplanes in $\p(R_{d+i}^*)$ containing
$\p(I_{d+i}^*)$ and the $(d+i-k)$-th osculating space to
$v_{d+i}(\p^n)$ at the point $[L^{d+i}]$.
In the dual setting it means that these hyperplanes define
forms  of degree $d+i$ in $(I_{d+i}^{\perp})^*$ with multiplicity  $(d+i-k+1)$ at $[L^{\vee}]$.
\\
To summarize, the number of non trivial Laplace equations is $\hh^0(K\otimes \mathscr O _{L^k})-s-N^+_s$ and
$\mathrm{coker}(\mathrm{H}^0(\Phi_{I,L^k}))\simeq (I_{d+i}^{\perp})^*\cap \mathrm{H}^0(\cI_{L^{\vee}}^{d+i-k+1}(d+i)).$
\end{proof}

\section{Remarks on failing SLP at the range  $2$ and line arrangements on $\p^2$}
\label{arrang_sec}
In this section we explain the connection with certain line arrangements in the plane and we add some remarks on \cite[Section 7]{DIV}.
\par
A line arrangement  is a collection of distinct lines in the projective plane. Arrangements of lines, and more generally of hyperplanes, have long been an important topic of study (see \cite{Cartier} or \cite{OT} for a good introduction).
\begin{defi}
Given a line arrangement, let us denote by
$f=0$ the equation of the union of lines of this arrangement. The
vector bundle $\mathcal{D}_0$ defined as the kernel of the  jacobian map:
$$ \begin{CD}
    0 @>>> \mathcal{D}_0 @>>> \mathscr O _{\p^2}^{3} @>(\partial f)>> \mathscr O _{\p^2}(d-1)
   \end{CD}
$$
 is called \em derivation bundle\em\ (or  logarithmic bundle ) of the line arrangement  (see \cite{S}  and \cite{Sc} for an introduction to
derivation bundles).
\\
One can consider the  lines of the arrangement in $\p^{2}$ as a set of distinct points $Z$  in $\p^{2\vee}$. Then we will
denote by  $\mathcal{D}_0(Z)$ the associated derivation bundle.

The arrangement of lines is said {\it free with exponents} $(a,b)$ if its derivation bundle splits on $\p^2$ as a sum of two line bundles, more precisely if
$$ \mathcal{D}_0(Z)=\mathscr O _{\p^2}(-a)\oplus \mathscr O _{\p^2}(-b),$$
while $(a,b)$ is called the \em general splitting type\em\ if it corresponds to the splitting of $\mathcal{D}_0(Z)$ over a general line $l\subset \p^2$.
\end{defi}
The general splitting type $(a,b)$ is related to the existence of curves of degree $a+1$ passing through $Z$, having multiplicity $a$ at $l^{\vee}\in \p^{2\vee}$.
More precisely,
\begin{lem}(\cite{V_AnnTol}, \cite[Proposition 2.1]{FV})
\label{linksd} Let $Z\subset \p^{2\vee}$ be a set of $a+b+1$ distinct points with $1\le a\le b$  and $l$ be
a general line in $\p^{2}$. Then  the following conditions are equivalent:
\begin{enumerate}
 \item $\mathcal{D}_0(Z)\otimes
\mathscr O _{l}=\mathscr O _{l}(-a)\oplus \mathscr O _{l}(-b)$.
 \item  $\mathrm{h}^0((\mathcal{J}_{Z}\otimes \mathcal{J}_{l^{\vee}}^{a})(a+1))\neq
0$ and $\mathrm{h}^0((\mathcal{J}_{Z}\otimes \mathcal{J}_{l^{\vee}}^{a-1})(a))=
0.$
\end{enumerate}
\end{lem}
So the general splitting type is related to the degree of suitable singular curves through $Z$ and, thanks to the result of Emsalem-Iarrobino (\cite{EI}) here stated in Theorem \ref{EI}, in
\cite[Proposition 7.2]{DIV} an equivalence between unstability of the derivation bundle and the failing of SLP at range $2$ is given. Actually, in the statement of Proposition 7.2 and its corollaries in \cite{DIV} the hypothesis on the non-existence of $0$-syzygies, that is implicitly used in the proofs, is missing. Here we give a precise statement, in a more general form with respect to the number of lines and the splitting type of the derivation bundle. In this way we determine an interval of possible degrees of generators of ideals failing SLP.
%
%
\begin{prop}\label{prop_bundle}
\label{th5}
 Let $I\subset R=\C[x,y,z]$ be an  artinian ideal generated by $2d+1+n$  polynomials $l_1^d, \ldots, l_{2d+1+n}^d$ where $n\geq 0$ and $l_i$ are distinct linear forms in $\p^2$.
Let  $Z=\{l_1^{\vee},\ldots, l_{2d+1+n}^{\vee}\}$ be the corresponding set of points in $\p^{2\vee}$. If the ideal is minimally generated in degree $d$, then the following conditions are equivalent:
\begin{enumerate}
\item The ideal $I$ fails the $\mathrm{SLP}$ at the range  $2$ in degree $d-2$.
\item The derivation bundle $\mathcal{D}_0(Z)$ is non-balanced with splitting type $(d-s,d+s+n)$, with $s\geq 1$.
\end{enumerate}
Moreover, if $n$ is even the following third condition is equivalent to the previous two:
\begin{enumerate}[3.]
\item
The derivation bundle $\mathcal{D}_0(Z)$ is unstable with splitting type $(d-s,d+s+n)$, with $s\geq 1$.
\end{enumerate}
\end{prop}
\begin{proof}
It is enough to argue as in \cite[Proposition 7.2]{DIV}.  The ideal $I$ fails the $\mathrm{SLP}$ at the range  $2$ in degree $d-2$ if and only if there exists a curve of degree $d$ through $Z$ and with multiplicity $d-1$ at a general point $P$. By Lemma \ref{linksd}, this condition is equivalent to asking that $\mathcal{D}_0(Z)$ has splitting type $(a,b)$  with $a\leq d-1$, so we can write $a=d-s$ with $s \geq 1$. As the number of points in a line arrangement is always $a+b+1$ we get $b=2d+n-a=d+s+n$ and the difference $|b-a|=2s\geq 2$, so $\mathcal{D}_0(Z)$ has to be non-balanced.
The last part comes from the  equivalence between unstability and the non-balanced condition $|b-a|\geq 2$ that holds when the first Chern class $a+b$ is even.
\end{proof}
Now, the corollary in \cite{DIV} becomes the following.
\begin{prop}\label{prop_arr}
 Let $\mathcal{A}=\{ l_1, \ldots, l_{a+b+1}\}$ be a free line arrangement with exponents $(a,b)$ such that  $a\le b $,  $b-a\ge 2$, let $Z=\{l_1^{\vee},\ldots, l_{a+b+1}^{\vee}\}$ be the corresponding set of points in $\p^{2\vee}$. For every integer $d$ such that $a+1\leq d\leq \lceil \frac{a+b}2\rceil$\footnote{By $\lceil x\rceil$ we denote the minimum integer greater or equal than $x$}, if the ideals defined below are minimally generated in degree $d$, then
  \begin{enumerate}
\item  if $a+b$ is even, $I=(l_1^d, \ldots, l_{a+b+1}^d)$ fails the SLP at the range $2$ and degree $d-2$;
\item if $a+b$ is odd, let $P$ be a point in general  position with respect to $Z$, then $I=(l_1^{d}, \ldots, l_{a+b+1}^{d},{P^\vee}^{d})$ fails the SLP at the range $2$ and degree  $d-2$.
    \end{enumerate}
\end{prop}
%
%
Now, we have to explain how we recognized that in \cite{DIV} the hypothesis on the minimality of generators has been forgotten. We thank the authors of \cite{CHMN16} for the following example.
\begin{ex}\label{CHMN_ex}
Let $\mathcal A=xyz(x+z)(x+2z)\prod_{j=1}^{12}(y+jz)$. It is a free arrangement of splitting type $(3,13)$. The derivation bundle is unstable because it is non-balanced and the first Chern class is even, but the ideal $I=(x^8,y^8,z^8,(x+z)^8,(x+2z)^8,(y+jz)^8|1\leq j\leq 12)$ has the SLP at range $2$.
\end{ex}
\begin{figure}[h!]
    \centering
    \includegraphics[height=4cm]{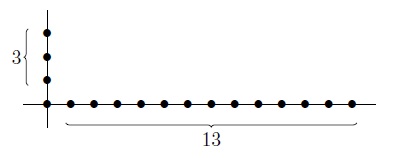}
    \caption{Dual points of the lines in Example \ref{CHMN_ex}}
  \end{figure}
 Actually, by investigating on the geometric meaning of $0$-syzygies in artinian ideals generated by powers of linear forms, we recognized that they are equivalent to the exisistence of a suitable number of aligned points in $Z$. So examples analogous to Example \ref{CHMN_ex} are the only possibility in order to get $0$-syzygies. In order to state precisely and prove this result,  we recall the following one by Ellia and Peskine \cite[Proposition pag. 112]{EP} related to the numerical character of $Z$.
\begin{defi}\label{defi}Let $S=\mathbb K[x_0,x_1]$ and $Z$ be a $0-$dimensional scheme in the projective plane $\mathbb P^2$. The \em numerical character \em\ of $Z$ is the sequence $(n_0,\ldots,n_{s-1})$ with $n_0\geq \cdots \geq n_{s-1}$ such that
$$0\rightarrow \oplus_{i=0}^{s-1} S(-n_i) \rightarrow \oplus_{i=0}^{s-1} S(-i) \rightarrow \frac{\mathbb K[x_0,x_1,x_2]}{I_Z}\rightarrow 0$$
is a minimal resolution and
\begin{enumerate}
\item $s$ is the minimal degree of a curve containing $Z$;
    \item $n_i \geq s$, for each $i=0,\ldots,s-1$;
        \item\label{huno} Defined $(a)_+= \max\{a,0\}$, $h^1(\mathscr I_Z(n))=\deg(Z)-H_Z(n)=\sum_{0}^{s-1}(n_i-n-1)_+-\sum_{i=0}^{s-1}(i-n-1)_+$, in particular $\deg(Z) = \sum_{i=0}^{s-1}(n_i-i)$.
\end{enumerate}
\end{defi}
\begin{lem}\label{lemma}
Let $Z\subset \mathbb P^2$ a set of points with numerical character $(n_i)_{i=0,\ldots,s-1}$. If $t$ is an integer such that $n_{t-1}>n_t+1$, then there exists a curve $C$ of degree $t$ such that the points of $Z$ contained in $C$ form a set with numerical character $(n_0,\ldots,n_{t-1})$.
\end{lem}
 \begin{thm}\label{th_aligned}
Let $l_1,\ldots,l_{2d+1}$ be linear forms in $\mathbb K[x_0,x_1,x_2]$. Then $$\dim_\mathbb K (l_1^d,\ldots,l_{2d+1}^d)_d<2d+1$$ if and only if in the set $Z=\{l_1^\vee,\ldots,l_{2d+1}^\vee\}\subset {\p^2}^*$ there are at least $d+2$ aligned points .
\end{thm}
\begin{proof}In the Veronese map $v_d$, aligned points  $l_1^\vee,\ldots,l_r^\vee\in L$ go to points $l_1^d,\ldots,l_r^d\in v_d(L)=C_d$, where $C_d \subset \mathbb P^d$ is the rational normal curve of degree $d$; so $d+1$ of them are linearly independent, while $d+2$ are dependent. So if there are $d+2$ aligned points we get $\dim_\mathbb K (l_1^d,\ldots,l_{2d+1}^d)_d<2d+1$.
 \\
 Conversely, let $\dim_\mathbb K (l_1^d,\ldots,l_{2d+1}^d)_d<2d+1$. By Emsalem-Iarrobino result (\cite{EI}, here Theorem \ref{EI}) we get
 $$
 \dim_{\mathbb K} {I_Z}_d=\dim_{\mathbb K} {\frac R {(l_1^d,\ldots,l_{2d+1}^d)}}_d.
 $$
   So the Hilbert function of $Z$ in degree $d$ is $H_Z(d)=\dim_{\mathbb K}(l_1^d,\ldots,l_{2d+1}^d)_d <2d+1$ and, by \eqref{huno} in Definition \ref{defi}, $h^1(\mathscr I_Z(d))=deg(Z)-H_Z(d)>2d+1-(2d+1)= 0$. With a simple calculation, we deduce that $n_0\geq d+2$. If $n_1\geq d+1$ we would have $2d+2>2d+1=\deg(Z)\geq n_0+n_1-1\geq d+2+d+1-1=2d+2$ and this is a contradiction. So $n_1<n_0-1$ and we can apply Lemma \ref{lemma} with $t=1$ in order to get a line containing a subset $Z'\subset Z$ consisting of $\deg Z'=n_0\geq d+2$ points.
 \end{proof}
So the hypothesis of non-existence of $0$-syzygies is equivalent to the condition that there are at most $d+1$ aligned points in $Z$. In Example \ref{CHMN_ex} there are too many aligned points, actually $14$ that are exactly $14=d+1+s$. Precisely, for $d=8$ in order to avoid linear dependence among $l_1^8,\ldots, l_r^8$, we can have at most $r=9$ aligned points in the dual set  $l_1^\vee,\ldots, l_r^\vee$. But there are $14$ aligned points so there are 5 more, corresponding to 14-9=5 independent $0$-syzygies.

\par With this interpretation of $0$-syzygies we can restate \cite[Proposition 7.2]{DIV} and its corollaries in a nicer geometric way.
\begin{prop}\label{prop_bundle2}
\label{th52}
 Let $I\subset R=\C[x,y,z]$ be an  artinian ideal generated by $\mathbf{2d+1}$  polynomials $l_1^d, \ldots, l_{2d+1}^d$ where $l_i$ are distinct linear forms in $\p^2$.
Let  $Z=\{l_1^{\vee},\ldots, l_{2d+1}^{\vee}\}$ be the corresponding set of points in $\p^{2\vee}$. If \textbf{there are no more than $d+1$ aligned points in $Z$}, then the following conditions are equivalent:
\begin{enumerate}
\item The ideal $I$ fails the $\mathrm{SLP}$ at the range  $2$ in degree $d-2$.
\item The derivation bundle $\mathcal{D}_0(Z)$ is unstable with splitting type $(d-s,d+s)$, with $s\geq 1$.
\end{enumerate}
\end{prop}
We recall also Terao's conjecture.
\begin{defi}
Let $\mathcal A$ be a line arrangement, then the combinatorics of $\mathcal A $ is the intersection lattice with reverse order.
\end{defi}\begin{conj}(Terao)
Let $\mathcal A$ be a free arrangement and $\mathcal A'$ an arrangement with the same combinatorics as $\mathcal A$. Then $\mathcal A'$ is free, too (i.e. the freeness is a combinatorial property).
\end{conj}
This conjecture for line arrangements has been proved up to $12$ lines (\cite{FV}).
 In \cite{DIV}, where freeness and unstability are related with failing SLP, an equivalent conjecture
was given in terms of SLP, that here we complete with the missed hypothesis on syzygies. 
\begin{conj}\cite{DIV}\label{conj_terao}
Let $\ell_1\cdots \ell_{2b+1}$ and $h_1\cdots h_{2b+1}$ be two arrangements with the same combinatorics and such that the dual sets of points have at most $b+1$ aligned points.
If $I=({{\ell}_1}^b,\ldots,{{\ell}_{2b+1}}^b)$ has SLP at range 2 in degree $b-2$ then also $J=(h_1^b,\ldots,h_{2b+1}^b) $ has SLP at range $2$ in degree $b-2$.
\end{conj}
Now, we end this note with a remark that links Example \ref{CHMN_ex} with Theorem \ref{thgen}.
If there is no $0$-syzygy, $N^- =0$ and the existence of a suitable singular curve is enough to the failure of SLP. Moreover, this hypothesis on $0$-syzygies is necessary, otherwise the unstability is not enough to get an ideal that fails SLP: in Example \ref{CHMN_ex} $H_{R/I}(8)=33=45-(17-5)$ that means that there are $s=5$ independent $0$-syzygies and the ideal has SLP and, moreover, the difference between Hilbert functions is $H_{R/I}(8)-H_{R/I}(6)=-5$. From this we got the idea to generalize Theorem \ref{th1bis} when there are $s$ syzygies in degree $i$ and no syzygy in degree $i-k$ just by replacing the integer $N(r,i,k,d)$ defined in \cite{DIV} with $N(r,i,k,d,s)$ introduced in Definition \ref{def_Ns}. In Example \ref{CHMN_ex} we get $N^-_s=5$ and the existence of singular curves is expected (in the sense of \cite{CHMN16}) and is not enough to the failure of SLP, by applying Theorem \ref{thgen}.
\\ We recall here the definition of unexpected curve.
\begin{defi}(\cite[Definition 2.1]{CHMN16})
We say that a reduced finite set of points $Z\subset \mathbb P^2$ admits an unexpected
curve if there is an integer $j > 0$ such that, for a general point $P$, $jP$ fails to impose the
expected number of conditions on the linear system of curves of degree $j + 1$ containing $Z$.
That is, $Z$ admits an unexpected curve of degree $j + 1$ if
$$
 h^0((\mathcal I_Z \otimes {\mathcal I_P}^j )(j + 1)) > \max \left\{h^0(\mathcal I_Z(j + 1)) -\left( \begin{array}{c}
j + 1\\
2
\end{array}\right)
; 0\right\}.
$$
\end{defi}
A very simple calculation shows that for $k=2,i=0, d=j+1$ we have $$N^-_s=\max \left\{h^0(\mathcal I_Z(j + 1)) -\left( \begin{array}{c}
j + 1\\
2
\end{array}\right)
; 0\right\},$$ so that we have the following special case of \cite[Theorem 6.5]{CHMN16}.
\begin{cor}\label{cor_unexp}
Let $I=(l_1^d, \ldots, l_{2d+1}^d)\subset R=\C[x,y,z]$ be an  artinian ideal generated by $2d+1$ distinct powers of linear forms in $\p^2$.
Let  $Z=\{l_1^{\vee},\ldots, l_{2d+1}^{\vee}\}$ be the corresponding set of points in $\p^{2\vee}$.  If there are no more than $d+1$ aligned points in $Z$, then the following are equivalent:
\begin{enumerate}[1.]
\item $Z$ has an unexpected curve of degree $d$;
\item $I$ fails the SLP in range $2$ and degree $d-2$;
\item the variety $\pi_{I_d}(v_d(\p^2))$ satisfies at least one non trivial Laplace equation of order $d-2$.
\end{enumerate}
\end{cor}
%
%
%
%
%
%

\section*{References}
\begin{thebibliography}{elsarticle-num}
\bibitem{BK} {H.~Brenner and A.~Kaid},  Syzygy bundles on $\p^2$ and the Weak Lefschetz Property. {\it  Illinois J. Math.}, 51 (2007), 1299--1308.
\bibitem{Cartier} {P.~Cartier}, Les arrangements d'hyperplans: un chapitre de g\'eom\'etrie combinatoire. {\it S\'{e}minaire Bourbaki}, Tome 23 (1980-1981) , Exposé no. 561 , p. 1-22
\bibitem{CHMN16} D.~Cook II, B.~Harbourne, J.~Migliore, U.~Nagel, \textit{Line arrangements and configurations of points with an unusual geometric property} (2017), arXiv:1602.02300v2
\bibitem{CN13} D.~Cook II, U.~Nagel, \textit{Enumerations of lozenge tilings, lattice paths, and perfect matchings and the weak Lefschetz property},(2013), arXiv:1305.1314
\bibitem{DI}  P.~De~Poi, G.~Ilardi, \textit{On higher Gauss maps}, Journal of Pure and Applied Algebra, vol. 219 (2015), no. 11, 5137—5148, doi: 10.1016/j.jpaa.2015.05.009.
\bibitem{DIV} R.~Di~Gennaro, G.~Ilardi, J.~Vall\`{e}s, \textit{Singular hypersurfaces characterizing the Lefschetz properties}  J. Lond. Math. Soc. (2) 89 (2014), no. 1, 194-212.
\bibitem{EI} J.~Emsalem, A.~Iarrobino, \textit{Inverse System of a Symbolic Power, I}, J. of Algebra 174, (1995) 1080--1090.
\bibitem{EP} Ph.~Ellia, Ch.~Peskine, \textit{Groupes de points de {${\mathbb P}^2$}: caract\`ere et
              position uniforme}, Algebraic geometry ({L}'{A}quila, 1988) Lecture Notes in Math. 1417 (1990), 111--116.
\bibitem{FV} D.~Faenzi, J.~Vall\`{e}s. \textit{Logarithmic bundles and line arrangements, an approach via the standard construction}, J. London Math. Soc. (2) 90 (2014), 675--694.
\bibitem{MMO} E.~Mezzetti, R.~Mir\'o-Roig and G.~Ottaviani, \textit{Laplace Equations and the Weak Lefschetz Property}, Canad. J. Math. 65(2013), 634-654.
\bibitem{OT} P.~Orlik and H.~Terao, {\it Arrangement of hyperplanes}, volume 300 of {\it Grundlerhen der Mathematischen Wissenschaften [Fundamental Principles of Mathematical Sciences].} Springer-Verlag, Berlin,  1992
\bibitem{S} {K.~Saito}, Theory of logarithmic differential forms and logarithmic vector fields. {\it J. Fac. Sci. Univ. Tokyo Sect. IA Math.}, 27(2) (1980), 265--291.
\bibitem{Sc} {H.~ K. Schenk},  Elementary modifications and line configurations in $\p^2$.
{\it  Comment. Math. Helv.} 78(3) (2003), 447--462.
\bibitem{T} {E.~Togliatti}, Alcune osservazioni sulle superficie razionali che rappresentano equazioni di Laplace, {\it Ann. Mat. Pura Appl.} 25 (1946), 325–339.
%
%
%
%

\bibitem{V_AnnTol} {J.~Vall\`es}, Fibr\'es logarithmiques sur le plan projectif. {\it Ann. Fac. Sci. Toulouse Math.} 16(2):385--395, 2007.
%
\end{thebibliography}
\section*{Acknowledgement} We thank Jean Vall\`{e}s for his valuable comments, David Cook II, Brian Harbourne, Juan Migliore and Uwe Nagel for their remarks on \cite{DIV} and the referee for her/his careful and precise reading and her/his suggestions.

%
%
\end{document}